\documentclass[11pt]{amsart}
\usepackage{amssymb, amsmath, amsthm, color, marginnote, hyperref, cancel ,amscd, amsfonts,mathtools}
\usepackage[nobysame]{amsrefs}

\usepackage{enumitem}
\usepackage[all]{xy}
\usepackage{hyperref}
\usepackage[T1]{fontenc} 
\usepackage{mathrsfs}
\usepackage{wasysym}

\usepackage{multicol}
\usepackage{rotating}
\usepackage{graphicx,fancyhdr}
\usepackage{multirow}
\usepackage[normalem]{ulem}

\hypersetup{colorlinks=true,citecolor=cyan,linkcolor=blue,linktocpage=true}

\newtheorem{theorem}{Theorem}

\newtheorem{lemma}[theorem]{Lemma}
\newtheorem{cor}[theorem]{Corollary}

\newtheorem{prop}[theorem]{Proposition}

\theoremstyle{definition}
\newtheorem{question}[theorem]{Question}
\newtheorem{conjecture}[theorem]{Conjecture}

\newtheorem{remark}[theorem]{Remark}
\newtheorem*{remark*}{Remark}
\newtheorem{definition}[theorem]{Definition}
\newtheorem{example}[theorem]{Example}
\newtheorem*{example*}{Example}

\newtheorem{problem}[theorem]{Problem}

\newcommand{\pf}{\begin{proof}}
\newcommand{\epf}{\end{proof}}

\newcommand{\yd}[1]{{}^{ #1 }_{ #1 }\mathcal{YD}}

\newcommand{\ydg}[2]{{}^{ #1 }_{ #1 }\mathcal{YD}^{#2}}

\newcommand{\cero}{\mathbf{0}}

\newcommand{\leftdual}{\star,\monoid}
\newcommand{\rightdual}{\monoid,\star}

\newcommand{\ad}{\operatorname{ad}}

\newcommand{\ord}{\operatorname{ord}}

\newcommand{\car}{\operatorname{char}}

\newcommand{\hZ}{\widehat{\mathcal{Z}}}

\newcommand{\gr}{\operatorname{gr}}

\newcommand{\GK}{\operatorname{GKdim}}

\newcommand{\supp}{\operatorname{supp}}

\newcommand{\corad}{\operatorname{corad}}

\newcommand{\C}{{\mathbb C}}
\newcommand{\N}{{\mathbb N}}

\newcommand{\ku}{{\Bbbk}} 

\newcommand{\superqa}[3]{{\bf A}_{#1}(#2|#3)}
\newcommand{\superda}[1]{{\mathbf D}(2,1;#1)}

\newcommand{\toba}{{\mathscr B}} 
\newcommand{\bq}{\mathfrak{q}}

\newcommand{\wtoba}{\widetilde{\toba}}

\newcommand{\Gb}{\mathbb G}

\newcommand{\Bc}{\mathcal B}

\newcommand{\Ec} {\mathcal E}

\newcommand{\Kc}{\mathcal K}
\newcommand{\Lc}{\mathcal L}
 
\newcommand{\Pc}{\mathcal P}

\newcommand{\g}{\mathfrak g}

\newcommand{\pref}{\mathfrak{Pre}_{\textrm{fGK}}}
\newcommand{\pre}{\mathfrak{Pre}}
\newcommand{\post}{\mathfrak{Post}}
\newcommand{\postn}{\mathfrak{Post}_{\textrm{Noeth}}}

\newcommand{\monoid}{\mathtt M}%{\varPhi}

\newcommand{\I}{\mathbb I}

\begin{document}
\title[Noetherian pointed Hopf algebras]{On Noetherian pointed Hopf algebras}
\author[Andruskiewitsch]{Nicol\'as Andruskiewitsch}
\address[Andruskiewitsch]{CIEM-CONICET.  Profesor Em\'erito, Facultad de Matem\'a\-tica, F\'isica,
Astronom\'\i a y Computaci\'on, Universidad Nacional de C\'ordoba. 
Medina Allende s/n (5000) Ciudad Universitaria, C\'ordoba, Argentina
\vspace{10pt}\newline 
Department of Mathematics and Data
Science, Vrije Universiteit Brussel, Pleinlaan 2, 1050 Brussels, Belgium
\vspace{10pt}\newline 
Shenzhen International Center for Mathematics, Southern University of Science and Technology, Shenzhen 518055, China}

\email{nicolas.andruskiewitsch@unc.edu.ar}

\author[Heckenberger]{István Heckenberger}
\address[Heckenberger]{Philipps-Universit\"at Marburg,
FB Mathematik und Infor\-matik,
Hans-Meer\-wein-Stra\ss e,
35032 Marburg, Germany.}
\email{heckenberger@mathematik.uni-marburg.de}

\author[Vendramin]{Leandro Vendramin}
\address[Vendramin]{Department of Mathematics and Data Science, Vrije Universiteit Brussel, Pleinlaan 2, 1050 Brussel, Belgium}
\email{Leandro.Vendramin@vub.be}

\begin{abstract}
In \texttt{arXiv:1405.4105} it was asked whether an affine Hopf algebra with finite Gelfand-Kirillov dimension is necessarily Noetherian.
It is well-known that the converse is not true--take the group algebra of a polycyclic group
which is not nilpotent-by-finite.
However, it was conjectured in \texttt{arXiv:2301.04428} 
that a \emph{pointed} affine Noetherian Hopf algebra whose group of group-likes is nilpotent-by-finite 
necessarily has finite Gelfand-Kirillov dimension.
In the present paper we conjecture that a post-Nichols algebra over a polycyclic-by-finite group
is Noetherian if and only if it is affine and has finite Gelfand-Kirillov dimension.
Partial results supporting this conjecture are presented; and their consequences for the preceding questions and conjectures is analyzed.
\end{abstract}

\thanks{ 
This work is partially supported by 
CONICET (PIP 11220200102916CO), FONCyT-ANPCyT (PICT-2019-03660), 
the Secyt (UNC) (Proyecto Consolidar 33620230100517CB),
and FWO--DFG (Weave) G043326N}

\maketitle

\section{Introduction} 

Throughout the text, $\ku$ will denote an algebraically closed field of characteristic $0$, unless explicitly stated otherwise. 
All algebras, coalgebras, etc., are over $\ku$.
Recall that a finitely generated algebra is called affine.
If $G$ is a group and $A$ is a (left) $G$-module algebra, then
$A \rtimes G$ denotes the semidirect, or smash, product of $A$ by $G$.

\medbreak
In this paper $J, H, K, M$ denote Hopf algebras with bijective antipode;
consequently, `$J$ Noetherian' means `$J$ left   Noetherian' or `$J$ 
right Noetherian', conditions that are equivalent.
Furthermore we assume that $H$ is pointed, $K$ is cosemisimple, 
and $M$ has the Chevalley property.

\medbreak
The authors of the preprint  \cite{arXiv:2507.23730} 
elaborate on the following conjecture:

\begin{conjecture}\label{conjecture:Noetherian-DME-GKdim}
\cite{brown-stafford}
If $H$ is  Noetherian pointed,  the following statements are equivalent:
\begin{enumerate}[leftmargin=7ex,label=\rm{(\roman*)}]
\item\label{item:GK-H-finita} $\GK H$ is finite ($\GK$ means Gelfand-Kirillov dimension).

\medbreak
\item\label{item:DME} $H$ satisfies the Dixmier-Moeglin Equivalence (DME).

\medbreak
\item\label{item:GK-GH-finita} The group $G(H)$ of group-like elements of $H$ is nilpotent-by-finite.
\end{enumerate}
\end{conjecture}

By the celebrated Gromov Theorem, 
it is clear that \ref{item:GK-H-finita} implies \ref{item:GK-GH-finita};
so,  a significant part of Conjecture \ref{conjecture:Noetherian-DME-GKdim} 
is the following question.

\begin{question}\label{question:Noetherian-finiteGK}
If $H$ is pointed Noetherian and $G(H)$ has finite growth, is 
$\GK H$ finite?
\end{question}

In the converse direction, it was asked:

\begin{question}\label{question:affine+finiteGK-Noetherian} \cite{Brown-Gilmartin-survey}
If $J$ is affine and $\GK J < \infty$, is $J$  Noetherian? 
\end{question}

One of the main purposes of \cite{arXiv:2507.23730} is to discuss recent advances 
on the following classical question (the converse 
is a theorem of Phillip Hall). 

\begin{conjecture}\label{conjecture:group algebra-Noetherian}
If the   group algebra $\ku\Gamma$ is Noetherian, then the  group  $\Gamma$ is polycyclic-by-finite.
\end{conjecture}

In this note we put together some known facts on, and discuss
potential approaches to, Questions \ref{question:Noetherian-finiteGK}
and \ref{question:affine+finiteGK-Noetherian} for 
pointed Hopf algebras beyond group algebras--assuming that Conjecture \ref{conjecture:group algebra-Noetherian} holds. 
Precisely, we relate these questions to analogous problems
on so-called post-Nichols algebras; namely, we propose  the following Conjecture.
\begin{conjecture} 
\label{conj:nichols-noeth}
Let $R$ be a post-Nichols algebra of $V \in \yd{\ku \Gamma}$ where
$\Gamma$ is a polycyclic-by-finite group.
The following statements are equivalent:
\begin{enumerate}[leftmargin=7ex,label=\rm{\textcolor{blue}{(\ref{conj:nichols-noeth}\alph*)}}]
\item\label{item:GK-R-finita} $R$ is affine and $\GK R$  is finite.

\medbreak 
\item\label{item:R-Noetherian} $R$ is left Noetherian.
\end{enumerate}
\end{conjecture}

\medbreak
The conjecture is related to Questions~\ref{question:Noetherian-finiteGK} and
\ref{question:affine+finiteGK-Noetherian},  
see Lemmas~\ref{lem:Noetherian-DME-GKdim} and \ref{lem:Noetherian-finiteGK}.

\medbreak
In Section \ref{sec:preliminaries} we review a number of general results on
Noetherianity and finite Gelfand-Kirillov dimension; in Section \ref{sec:corad-filtr} 
we collect results on the structure of pointed Hopf algebras related to the topics of the paper. 
The relations between  Questions \ref{question:Noetherian-finiteGK}
and  \ref{question:affine+finiteGK-Noetherian} and Conjecture~\ref{conj:nichols-noeth}
are discussed in Section \ref{sec:corad-graded}.

\medbreak
The focus of Section \ref{sec:diagonal-type} is post-Nichols algebras of diagonal type.
We answer Conjecture \ref{conj:nichols-noeth} affirmatively when $R$ is a Nichols algebra
of diagonal type, see Corollary \ref{cor:noeth-implies-finiteGK}.
One of the main results of this paper is Theorem~\ref{thm:Noetherian-finiteGK-diagtype}
that answers positively Question \ref{question:affine+finiteGK-Noetherian} 
when the infinitesimal braid is of diagonal type, assuming 
a few technical conditions. We believe that the ideas of its proof will be useful beyond
these conditions. 

\medbreak
We introduce the core, a new invariant of graded Hopf algebras in $\yd{J}$ which are semisimple
objects, in Section \ref{sec:core}. This is used to give a proof of Theorem~\ref{thm:Noetherian-finiteGK-diagtype}. 
We speculate on the additional steps that need to be taken to solve the Questions \ref{question:Noetherian-finiteGK} and \ref{question:affine+finiteGK-Noetherian}, as well as Conjecture \ref{conj:nichols-noeth} in  Section \ref{sec:remains}.
Some comments on the general, non-pointed, case are compiled in Section \ref{sec:chevalley}.

\subsection*{Conventions}  We set $\I_{\theta} \coloneqq \{1, \dots, \theta\}$.
The Nichols algebra of a braided vector space $V$ is denoted by $\toba(V)$.
The subalgebra of an associative algebra $A$ generated by $X \subset A$
is denoted by $\ku \langle X\rangle$.
For unexplained terminology and more details on the notions discussed here, 
see \cites{andrus-schneider,andrus-infinite,andrus-leyva,hs-book}.

\section{Preliminaries} \label{sec:preliminaries}

We shall need the following classical results. First, we shall use without
further warning that if $B$ is a subalgebra of a left Noetherian algebra $A$
and $A$ is a faithfully flat left $B$-module, then $B$ is left Noetherian as well.

\begin{prop}\label{prop:graded-Noetherian} \cite{mcconnell-robson}*{Theorem 1.6.9}
Let $A$ be a filtered ring. If $\gr A$ is left Noetherian, then so is 
$A$.
\end{prop}

Generally speaking, the converse of the preceding result
is not true and this is a crucial obstacle.

\begin{prop}\label{prop:smash-Noetherian} \cite{passman-infinite}*{Proposition 1.6}
Let $G$ be a polycyclic-by-finite group and let $R$ be a $G$-module algebra.
If $R$ is left Noetherian, then so is $R\rtimes G$.
\end{prop}

The proof of the following result is a variation of Hilbert's argument
for his Basissatz.

\begin{lemma}\label{lemma:graded-Noetherian} \cite{jia-zhang}
A graded connected left Noetherian
algebra  is affine
and its homogeneous components are finite dimensional.
\end{lemma}

Pointed Hopf algebras have a number of favorable properties. 
We list a few of them necessary for future discussions.

\begin{prop}\label{prop:antipode-pointed}
\cite{montgomery-book}*{Corollary 5.2.11}
The antipode of a pointed Hopf algebra is bijective.
\end{prop}

\begin{theorem}\label{thm:pointed free} \cite{radford77} 
Pointed Hopf algebras are free over their Hopf subalgebras.
\end{theorem}

Thus, if $H$ is pointed Noetherian, then any Hopf subalgebra, in particular
$\ku G(H) $, is also Noetherian.
We will also need the following braided version of Theorem~\ref{thm:pointed free}. Recall that a \emph{connected} Hopf algebra in $\yd{J}$ is a Hopf algebra with coradical $\ku$.

\begin{prop}\cite{AAGMV}*{Proposition 3.6}
\label{pro:braided_free}
Each connected Hopf algebra in $\yd{J}$ is a free left and right module over every right coideal subalgebra.
\end{prop}

\medbreak The following result gives a partial positive answer to a question from 
\cite{Wu-Zhang}.

\begin{theorem}\label{thm:pointed noeth-affine}
\cite{goodearl-zhang16} Pointed Noetherian Hopf algebras are affine.
\end{theorem}

We shall also need:

\begin{prop}\label{prop:pointed-affine-gp-fg}
\cite{zhuang}*{Corollary 3.5} 
The group of group-like elements of a pointed affine Hopf algebra
is finitely generated.
\end{prop}

In fact, a more general result is now known.

\begin{theorem}\label{thm:pointed noeth-affine2}
\cite{jia-zhang26} A Noetherian Hopf algebra is affine if
and only if, its Hopf coradical is affine.
\end{theorem}

\section{The coradical filtration} \label{sec:corad-filtr}
Let $(H_{n})_{n \geq 0}$ be the coradical filtration of the 
pointed Hopf algebra
$H$ and let $\gr H$ be the associated graded Hopf algebra. 
Set $\Gamma = G(H)$. Then $\gr H$ splits as the bosonization
\begin{align}\label{eq:splitting-grH}
\gr H \simeq R \# \ku \Gamma,
\end{align}
where $R$ is a strictly graded (i.e., coradically graded and connected)
Hopf algebra in $\yd{\ku\Gamma}$. The important invariant $R$ is called the \emph{diagram} of $H$.
A strictly graded Hopf algebra $R = \bigoplus_{n \geq 0} R^{n}$
is called a \emph{post-Nichols algebra} of $V \coloneqq R^{1}$.
In addition, the subalgebra of $R$ generated by $V$ 
is isomorphic to the Nichols algebra $\toba(V)$ of $V$. 

\medbreak
It would be crucial to have an answer to the following question, a partial converse of Proposition~\ref{prop:graded-Noetherian}.

\begin{question}\label{question:graded-Noetherian} 
Let $H$ be a Noetherian Hopf algebra. 
Under what conditions is $\gr H$ also Noetherian?
\end{question}

Note that $\gr H$ in Question~\ref{question:graded-Noetherian} 
is not necessarily Noetherian, see Example~\ref{exa:R-not affine}.
To highlight the ambition of Question~\ref{question:graded-Noetherian},  note that, 
for enveloping algebras, it reduces to a classical conjecture,
see \cites{Sierra-Walton,andrus-mathieu,buzaglo,mathieu} for recent advances:

\begin{conjecture}\label{conj:sierra-walton} \cite{Sierra-Walton}
Let $\g$ be a Lie algebra.
If the enveloping algebra $U(\g)$ is Noetherian, then $\dim \g$ is finite.
\end{conjecture}

Example \ref{exa:R-not affine} also shows that $H$ affine does not imply $\gr H$ affine; 
it is well known that the converse is true.
In contrast with Question \ref{question:graded-Noetherian}, one has:

\begin{prop}\label{prop:kl-graded-filtered}
\cite{krause-lenagan}*{Lemma 6.5}
Let $A$ be a filtered algebra and $\gr A$ the associated graded algebra. Then $\GK A \geq \GK \gr A$
and the equality holds when $\gr A$ is finitely generated.
\end{prop}

For pointed Hopf algebras there is a more precise  result. 

\begin{theorem}\label{thm:zhuang}\cite{zhuang}*{Theorem 5.4}, \cite{zhuang-jpaa}*{5.5}. 
If $R$ is affine, then 
\begin{equation}\label{eq:GK-grH}
\GK R + \GK \ku \Gamma = \GK \gr H = \GK H.
\end{equation}
\end{theorem}

\section{Coradically graded pointed Hopf algebras} \label{sec:corad-graded}
\subsection{On Conjecture ~\ref{conj:nichols-noeth}}

Recall that a coradically graded pointed Hopf algebra 
$L = \bigoplus_{n \geq 0} L^{n}$ 
splits as the bosonization
\begin{align}\label{eq:splitting-corad-graded}
L \simeq R \# \ku G(L),
\end{align}
where $R = \bigoplus_{n \geq 0} R^{n}$ is a post-Nichols algebra of $\Pc(R)=R^{1}$ in $\yd{\ku G(L)}$.

\begin{lemma}\label{lemma:pointed corad-graded}
Let $L$ be a coradically graded pointed Hopf algebra. 
If $L$ is Noetherian, then both the diagram $R$ of $L$ and $\ku G(L)$
are left Noetherian (hence, the Nichols algebra $\toba(R^1)$
is Noetherian as well). The converse holds provided that
Conjecture \ref{conjecture:group algebra-Noetherian} is true.
\end{lemma}

\pf Since $L$ is a free left  $R$-module, $R$ is left Noetherian; $\ku G(L)$, as well as $\toba(R^1) \# \ku G(L)$,
are Noetherian by Theorem \ref{thm:pointed free}. 
Hence $\toba(R^1)$ is left Noetherian.

Conversely, if $R$ is left Noetherian and  
$G(L)$ is  polycyclic-by-finite, then $L$ is Noetherian
by Proposition \ref{prop:smash-Noetherian}.
\epf

To deal with Questions \ref{question:Noetherian-finiteGK}
and  \ref{question:affine+finiteGK-Noetherian} it is
natural to state Conjecture~\ref{conj:nichols-noeth} from the Introduction; they are
related as follows.

\begin{lemma}
\label{lem:Noetherian-DME-GKdim}
Let $H$ be a pointed Noetherian Hopf algebra with $\Gamma=G(H)$ nilpotent-by-finite, 
and let $L=\gr H$. Assume 
that
\begin{enumerate}[leftmargin=*,label=\rm{(\roman*)}]
\item\label{item:L-noeth} $L$ is Noetherian, and 
\item\label{item:diag-finiteGK} the diagram $R$ of $H$ has finite $\GK$.
\end{enumerate}
Then $\GK H<\infty$. 
\end{lemma}

From the proof it will be clear that  if in Conjecture~\ref{conj:nichols-noeth} for $R$ property \ref{item:R-Noetherian} implies \ref{item:GK-R-finita} then \ref{item:diag-finiteGK} can be omitted.

\begin{proof}
As above, $L\simeq R\# \ku \Gamma$, where $R$ is the diagram of $H$. 
Moreover, $L$ is Noetherian by hypothesis \ref{item:L-noeth}.
By Lemma~\ref{lemma:pointed corad-graded}, $R$ is left Noetherian, 
hence it is affine by Lemma~\ref{lemma:graded-Noetherian}.
By hypothesis \ref{item:diag-finiteGK}, $\GK R < \infty$. Moreover, $H$ is affine by 
Theorem~\ref{thm:pointed noeth-affine}, $\Gamma$ is finitely generated by Proposition~\ref{prop:pointed-affine-gp-fg}, and is nilpotent-by-finite by assumption. Therefore $\GK \ku \Gamma < \infty$.
By Theorem~\ref{thm:zhuang}, 
\[
\GK H =  \GK R + \GK \ku \Gamma < \infty.\qedhere 
\]
\end{proof}

\begin{lemma}\label{lem:Noetherian-finiteGK}
Let $H$ be a pointed affine Hopf algebra with $\GK H < \infty$
such that the diagram $R$ of $H$ is affine. 
If \ref{item:GK-R-finita} implies \ref{item:R-Noetherian} for $R$, 
then $H$ is Noetherian.
\end{lemma}

\pf
Let  $\gr H \simeq R \# \ku \Gamma$  be as above. 
By Proposition \ref{prop:kl-graded-filtered}, $\GK \gr H < \infty$. 
Hence $\GK R < \infty$ and $\GK \ku \Gamma < \infty$. 
If \ref{item:GK-R-finita} implies \ref{item:R-Noetherian}, then $R$ is left Noetherian.
Proposition \ref{prop:pointed-affine-gp-fg} guarantees
that  $\Gamma$ is finitely generated.
Since $\GK\ku\Gamma<\infty$, it follows from Gromov Theorem that  $\ku\Gamma$ is nilpotent-by-finite, and hence polycyclic-by-finite. 
Thus $\gr H$ is Noetherian by Proposition \ref{prop:smash-Noetherian}, and $H$ is Noetherian by Proposition \ref{prop:graded-Noetherian}.
\epf

\subsection{Pre- and post-Nichols algebras}
We  shall approach post-Nichols algebras through pre-Nichols algebras.
For a wider range of applications, it is convenient to state the following vast generalisation of Conjecture~\ref{conj:nichols-noeth}, which includes the generalisation of 
Conjecture~\ref{conj:nichols-noeth} with arbitrary group algebras.

\begin{conjecture}\label{conj:nichols-noeth-gral} 
Let $J$ be a Hopf algebra. Let $R$ be a post-Nichols algebra of  
$V \in \yd{J}$.
The following statements are equivalent:
\begin{enumerate}
[leftmargin=7ex,label=\rm{\textcolor{blue}{(\ref{conj:nichols-noeth-gral}\alph*)}}]
\item\label{item:GK-R-finita-gral} $R$ is affine and $\GK R$  is finite.

\medbreak 
\item\label{item:R-Noetherian-gral} $R$ is left Noetherian.
\end{enumerate}
\end{conjecture}

Notice that \ref{item:R-Noetherian-gral} implies 
that $R$ is affine and the homogeneous components of $R$ 
are finite-dimensional by Lemma \ref{lemma:graded-Noetherian},
hence $\dim V < \infty$. So we may assume this last condition.

\medbreak 
Recall that a pre-Nichols algebra of $V$ is a graded connected Hopf algebra $\Bc = \bigoplus_{n \geq 0} \Bc^n \in \yd{J}$
such that $\Bc^1 \simeq V$ generates the algebra $\Bc$. 
Thus there is a surjective map $\Bc \twoheadrightarrow\toba(V)$
of graded   Hopf algebras in $\yd{J}$.

\medbreak
Conversely, if $R= \bigoplus_{n \geq 0} R^{n}$ is a post-Nichols algebra of  $V$, then
there is an inclusion $\toba(V) \hookrightarrow R$ of graded  
Hopf algebras in $\yd{J}$ and the graded dual \footnote{Here we mean \emph{right} dual.}
$R^d= \bigoplus_{n \geq 0} (R^{n})^*$
is a pre-Nichols algebra of $V^*$.

\begin{lemma}\label{lema:pre-post}\cite{AAH1}*{Lemma 3.3}
Let $J$ be a Hopf algebra, 
let $\Bc$ be a pre-Nichols algebra of $V \in \yd{J}$ with $\dim V < \infty$,
and let $R \coloneqq \Bc^d$ be the graded dual of $\Bc$. 
Then $\GK R \leq \GK \Bc$; if $R$ is affine,
then the equality holds.
\end{lemma}

\medbreak
To deal with the implication
\ref{item:GK-R-finita-gral} $\implies$ \ref{item:R-Noetherian-gral},
the following notion is useful, see  \cite{andrus-sanmarco}*{Subsection 2.5} for details.
Let $V \in\yd{J}$, $\dim V < \infty$. 
Let $\pre(V)$ be the set of all  pre-Nichols algebras of $V$ in 
$\yd{J}$. 
This is a poset with respect to
the order $\Bc \leq \Bc'$ if and only if there exists a morphism 
$\Bc \to \Bc'$ of graded Hopf algebras in $\yd{\ku \Gamma}$ being an isomorphism in degree one. In particular, such morphisms are surjective. Moreover, $\Bc \leq \toba(V)$ for any $\Bc \in \pre(V)$. Hence $\toba(V)$ is a top element of $\pre(V)$. Let $\pref(V)$ be the subposet of $\pre(V)$ consisting of the pre-Nichols algebras with finite $\GK$.  It is an upper set of $\pre(V)$.

\medbreak
Dually let $\post(V)$ 
be the set of all  post-Nichols algebras of $V$ in 
$\yd{J}$.  
This is a poset with respect to
$\Bc \leq \Bc'$ if and only if there exists a morphism 
$\Bc \to \Bc'$ of graded Hopf algebras in $\yd{\ku \Gamma}$ being an isomorphism in degree one. In particular, such morphisms are injective. Moreover, $\toba(V)\leq \Bc$ for any 
$\Bc \in \post(V)$. Hence $\toba(V)$ is a bottom element of $\post (V)$. Let $\postn(V)$ be the subposet of $\post(V)$ consisting of the left Noetherian post-Nichols algebras.
It is a lower set of $\post (V)$.

\medbreak
Note that $\Bc \in \pre (V)$ if and only if $\Bc^d\in \post (V^*)$.

\begin{cor}\label{cor:pre-post}
Let $V \in\yd{J}$, $\dim V < \infty$. Assume that 
\begin{align}\label{eq:assumption}
\Bc^d \in \postn(V) \text{ for all } \Bc \in \pref(V^*).
\end{align}
If $R$ is an affine post-Nichols algebra of $V$ with finite $\GK$, then $R$
is left Noetherian. 
\end{cor}

\pf  If   $R$ 
satisfies \ref{item:GK-R-finita-gral}, then the pre-Nichols
algebra $\Bc = R^d$ has finite $\GK$ by Lemma \ref{lema:pre-post}.
By the assumption \eqref{eq:assumption}, $R \simeq \Bc^d$ is Noetherian.
\epf

\medbreak
Following \cite{andrus-sanmarco}, $\Ec \in \pref(V)$ is said to be \emph{eminent}
if $\Ec$ is a bottom element of $\pref(V)$, that is, if 
$\Ec \leq \Bc$ for any $\Bc \in \pref(V)$. There are Yetter-Drinfeld
modules that do not admit an eminent pre-Nichols algebra.

\begin{lemma}
\label{lemma:eminent-Noetherian}
Let $V \in\yd{J}$ with $\dim V < \infty$. If
$V^*$ admits an eminent pre-Nichols algebra $\Ec$
whose graded dual $\Ec^d$ is Noetherian,
then $V$ satisfies \eqref{eq:assumption}. Hence
any affine post-Nichols algebra of $V$ with finite $\GK$ 
is left Noetherian. 
\end{lemma}

\pf Let $\Bc \in \pref(V^*)$. By hypothesis,
there exists a surjective morphism 
$\Ec \to \Bc$ of graded Hopf algebras in $\yd{J}$ such that $\Ec^d$ is Noetherian. Hence there is an injective morphism 
$\Bc^d \to \Ec^d$ of the same sort.
By Proposition~\ref{pro:braided_free}, $\Bc^d$ is Noetherian; 
thus \eqref{eq:assumption} holds and Corollary \ref{cor:pre-post} applies.
\epf

\section{Post-Nichols algebras of diagonal type} \label{sec:diagonal-type}
Here we discuss Conjectures \ref{conj:nichols-noeth}
and \ref{conj:nichols-noeth-gral} under the assumption that $V$ is of diagonal type.

 \subsection{Noetherian Nichols algebras of diagonal type}
 In this Subsection we show that Noetherian Nichols algebras of diagonal type
 have finite $\GK$. 
 In fact, we present a fairly general result that implies this.
 
 \medbreak
 Let $J$ be a Hopf algebra and  let $V\in\yd{J}$ be finite-dimensional and semisimple, that is,
 $V=\bigoplus_{i=1}^\theta V_i$ where each $V_i\in\yd{J}$ is finite-dimensional and simple. Then $\toba(V)$ is $\N_0^\theta $-graded.
 Let $\Kc(\toba(V))$ be the set of $\N_0^\theta$-graded right coideal subalgebras of $\toba(V)$, see \cite{hs-book}*{Definition~14.1.3}. 
 
 \medbreak
 A poset with no infinite strictly increasing chains is called a \emph{well‑founded poset}.
 
 \begin{lemma}\label{lema:well-founded}
 If $\toba(V)$ is Noetherian, then $\Kc(\toba(V))$ is well-founded.
 \end{lemma}
 
 \pf 
 It follows from \cite{hs-book}*{Theorem \,6.3.2} that there is a correspondence between 
 $\Kc(\toba(V))$ and the set of $\N_0^\theta$-graded coideal right ideals 
 respecting inclusion; since $\toba(V)$ is Noetherian, the latter
 is well-founded.
 \epf
 
 \begin{definition} \cite{hs-book}*{Definition 13.4.2}
 Let $i\in \I_{\theta}$. 
 We say that $V$ is $i$-finite if for
 all $j \in \I_{\theta}\backslash\{i\}$ there exists $m \in \N$ such that $(\ad V_i)^{m}(V_j) = 0$ 
 in $\toba(V)$.
 \end{definition}
 
 \begin{lemma}\label{lema:i-finite}
 If $\Kc(\toba(V))$ is well-founded, then $V$ is $i$-finite for all $i\in \I_{\theta}$.
 \end{lemma}
 
 \pf
 If $V$ is not $i$-finite, then there is $j\in \I_{\theta}\setminus \{i\}$ such that $(\ad V_i)^n(V_j)\ne 0$ for all $n \geq 0$. Let
 \[ E_n=\ku\langle (\ad\,V_i)^k(V_j)\mid 0\le k\le n\rangle  \in \Kc(\toba(V))\]
 for all $n\ge 0$.
 Then $(E_n)_{n\ge 0}$ is a strictly increasing  sequence of elements of $\Kc(\toba(V))$, and hence $\Kc(\toba (V))$ is not well-founded. 
 \epf
 
 Given $i\in \I_{\theta}$, $R_i(V)$ denotes the $i$-th reflection of $V$, which exists
 because $V$ is $i$-finite.
 
 \begin{lemma}
 \label{lem:reflections_are_well-founded}
 If $\Kc(\toba(V))$ is well-founded, then $\Kc(\toba(R_i(V)))$ is well-founded for all $i\in \I_{\theta}$.
 \end{lemma}
 
 \pf
 Let $i\in \I_{\theta}$. By the previous lemma, $V$ is $i$-finite. Assume that there is an infinite strictly increasing sequence $(E_n)_{n\ge 0}$ of elements of $\Kc(\toba(R_i(V)))$. We may assume that either none or else all of these elements contain $(R_i(V))_i$ as a subspace. 
 In the first case, let $(t_i(E_n))_{n\ge 0}$ be the infinite sequence in $\Kc(\toba(V))$ defined in  
 \cite{hs-book}*{Theorem \,14.1.4(1)}. By~\cite{hs-book}*{Theorem \,12.4.5}, 
 it is strictly increasing. 
 In the second case, by \textit{loc. cit.}, $(t_i^{-1}(E_n))_{n\ge 0}$ is an infinite strictly increasing sequence in $\Kc(\toba(V))$.
 \epf

 See  \cite{hs-book}*{Chapter 13} for the notion of root
 system. 
 
 \begin{lemma}
 \label{lem:well-founded->all_reflections}
 If $\Kc(\toba(V))$ is well-founded, then $V$ admits all reflections and the root system of $V$ is finite.
 \end{lemma}
 
 \pf
 The first part of the claim follows from Lemma~\ref{lem:reflections_are_well-founded}. 
 Assume that the root system of $V$ is infinite. Then there is an infinite sequence
 $(i_k)_{k\ge 1}$ of elements of $\I_{\theta}$ such that
 $(i_1,\dots,i_n)$ is $(V_1,\dots,V_\theta)$-reduced for all $n\ge 0$; 
 see~\cite{hs-book}*{Definition \,9.2.1}. 
 Let $E_n=E^{\toba(V)}(i_1,\dots,i_n)$ for all $n\ge 0$, 
 as defined in ~\cite{hs-book}*{Theorem \,14.1.9(7)}.  
 Then~\cite{hs-book}*{Theorem \,14.1.9(6)} 
 implies that $(E_n)_{n\ge 0}$ is a strictly increasing sequence 
 in $\Kc(\toba(V))$, and hence $\Kc(\toba(V))$ is not well-founded.
 \epf

 \begin{theorem}
 \label{th:noeth-implies-finiteRS}
 Let $J$ be a Hopf algebra and let $V=\bigoplus_{i=1}^\theta V_i\in\yd{J}$ be finite-dimensional and semisimple. 
 If $\toba(V)$ is Noetherian, then the root system of $\toba (V)$ is finite.
 \end{theorem}
 
 \pf Since $\toba(V)$ is Noetherian, $\Kc(\toba(V))$ is well-founded by Lemma \ref{lema:well-founded}.
 Then Lemma \ref{lem:well-founded->all_reflections} shows that $V$ admits all reflections and that
 the root system of $\toba(V)$ is finite. 
 \epf

 We focus now on the class of braided vector spaces of \emph{diagonal type},
 namely those pairs $(V, c^{\bq})$  where $V$ is a vector space with a basis $v_1,\dots ,v_\theta$,
 and $\bq=(q_{ij})\in(\Bbbk^{\times})^{\I_{\theta}\times\I_{\theta}}$ 
 is  the so-called braiding matrix, such that the braiding 
 $c^{\bq}: V\otimes V\to V\otimes V$ is given by 
 \begin{align} \label{eq:diag-br}
 c^{\bq}(v_i\otimes v_j)&=q_{ij}\,v_j\otimes v_i, & i,j&\in\I_{\theta}.
 \end{align}
 
 The following result, conjectured in~\cite{AAH2},  
 reduces to   \cite{heckenberger-adv} the classification of Nichols algebras of diagonal type with finite $\GK$.
 
 \begin{theorem}
 \cite{angiono-garciai-finiteGK}
 \label{thm:finiteGK}
 Let $(V, c^{\bq})$ be a braided vector space of diagonal type. 
 Then $\GK \toba(V) < \infty$ if and only if the corresponding root system is finite,  hence it 
 admits a PBW basis with a finite number of generators. 
 \end{theorem}
 
 We are ready to answer affirmatively to the Conjecture \ref{conj:nichols-noeth-gral} when $V$ is of diagonal type and $R = \toba(V)$.
 
 \begin{cor}\label{cor:noeth-implies-finiteGK}
 Assume that $V$ is a braided vector space of diagonal type. The following statements
 are equivalent: 
 \begin{enumerate}[leftmargin=*,label=\rm{(\roman*)}]
 \item\label{item:toba-noeth} $\toba(V)$ is Noetherian.
 \item\label{item:toba-finiteGK} $\toba(V)$ has finite $\GK$.
 \item\label{item:toba-finite-root-sys} The root system of $\toba(V)$ is finite.
 \end{enumerate}
 \end{cor}
 
 \pf 
 Theorem~\ref{th:noeth-implies-finiteRS} 
 shows that \ref{item:toba-noeth} implies \ref{item:toba-finite-root-sys}. The equivalence of the last two statements
 is Theorem~\ref{thm:finiteGK}. Finally, the proof of the implication
 \ref{item:toba-finite-root-sys} $\implies$\ref{item:toba-noeth} is essentially the same as the proof of \cite{angiono-prenichols}*{Theorem 19}.
 \epf

\subsection{Pre-Nichols algebras of diagonal type}

 \medbreak Now we turn to pre-Nichols algebras of diagonal type with finite $\GK$, 
 a path to Noetherian post-Nichols algebras as suggested 
 by Lemma \ref{lemma:eminent-Noetherian}.
By Theorem \ref{thm:finiteGK}, the corresponding root systems are finite.
 Many braided vector spaces of diagonal type with finite root system
 admit an eminent pre-Nichols algebra, as a consequence of \cites{andrus-sanmarco,ACSa,ACSa2,angiono-campagnolo}, see below.

In the following theorem we will consider Dynkin diagrams of the following form: 
\begin{align}\label{eq:exceptions}
& \begin{aligned} 
&\circ & & \text{one vertex with label $q=\pm 1$, or}
\\
&\circ & &\text{Cartan type $A_{\theta}$ with $\theta \ge 2$ and label $q=-1$, or}
\\
&\circ & &\text{Cartan type $D_{\theta}$ with $\theta \ge 4$ and label $q=-1$.}
\end{aligned}
\end{align}

\begin{theorem}\label{thm:eminent} \cite{angiono-campagnolo}*{Theorem 1.1; see also Remark 4.9}.
Let $V$ be a braided vector space of diagonal type such that $\dim \toba(V)< \infty$.
Assume that none of the connected components of the Dynkin diagram of $V$ 
is as in ~\eqref{eq:exceptions}. 
% are  
% \begin{align}\label{eq:exceptions}
% & \begin{aligned} 
% &\circ & & \text{neither one-dimensional with label $q=\pm 1$,}
% \\
% &\circ & &\text{nor of Cartan type $A_{\theta}$ with $\theta \ge 2$ and label $q=-1$,}
% \\
% &\circ & &\text{nor of Cartan type $D_{\theta}$ with $\theta \ge 4$ and label $q=-1$.}
% \end{aligned}
% \end{align}
Then $V$ admits an eminent pre-Nichols algebra $\Ec(V)$ which  fits into an exact sequence of 
graded braided Hopf algebras 
\begin{align}\label{eq:exact-sequence-prenichols}
\hZ  \hookrightarrow \Ec(V) \twoheadrightarrow \toba(V),
\end{align} where $\hZ $ is an affine algebra of $q$-polynomials. 
The distinguished pre-Nichols algebra $\widetilde{\Bc}(V)$, introduced in \cite{angiono-prenichols}, has  $\GK \widetilde{\Bc}(V) < \infty$ and is eminent
 except for the following  Dynkin diagrams:

\begin{enumerate}[leftmargin=5ex,label=\rm{(\roman*)}]
\item\label{item:eminent-not-distinguished-A2} \cite{andrus-sanmarco}  $A_2$ with $q\in \Gb_3'$.

\item\label{item:eminent-not-distinguished-Asuper-1} \cite{ACSa}  
$\superqa{3}{q}{\{2\}}$ and $\superqa{3}{q}{\{1,2,3\}}$ with $q\in \Gb_{N}'$.

\item\label{item:eminent-not-distinguished-g(2,3)-1} \cite{ACSa2} 
$\g(2,3)$ with diagrams  $\xymatrix@C=10pt{\overset{-1}{\circ} \ar@{-}[r]^{\xi} &\overset{-1}{\circ}\ar@{-}[r]^{\xi}\ &\overset{-1}{\circ}}$ and
$\xymatrix@C=10pt{\overset{-1}{\circ} \ar@{-}[r]^{\xi^2} &\overset{\xi}{\circ}\ar@{-}[r]^{\xi}\ &\overset{-1}{\circ}},$ $\xi \in \Gb_3'$.
\end{enumerate}
\end{theorem}

Above, $\Gb_{N}'$ is the set of primitive $N$-th roots of 1.
The braided vector spaces in \ref{item:eminent-not-distinguished-A2}, \ref{item:eminent-not-distinguished-Asuper-1} and \ref{item:eminent-not-distinguished-g(2,3)-1}
have an eminent pre-Nichols algebra $\Ec(V)$, see \cite{angiono-campagnolo}. 
When  the Dynkin diagram of $V$  is as in \eqref{eq:exceptions},
it is not known whether eminent pre-Nichols algebras exist.  
See  \cite{angiono-campagnolo}*{Remarks 4.9 \& 4.10}.

Now we turn to infinite dimensional Nichols algebras with finite $\GK$.

\begin{theorem}\cite{campagnolo}\label{thm:campagnolo}
Let $V$ be a finite-dimensional braided vector space of diagonal type with
connected Dynkin diagram such that $\dim \toba(V) = \infty$ 
but $\GK \toba(V) < \infty$. Then, either $\toba(V)$ is eminent, or the set of pre-Nichols algebras with finite $\GK$ has exactly two elements, namely $\toba(V)$ and $\Ec(V)$,
which occurs in the following cases:

\begin{enumerate}[leftmargin=5ex,label=\rm{(\roman*)}] \setcounter{enumi}{3}
\item\label{item:eminent-not-distinguished-Asuper-1-gen}   
$\superqa{3}{q}{\{2\}}$ and $\superqa{3}{q}{\{1,2,3\}}$ 
with $q\in \ku \backslash\Gb_{\infty}$.
 
\item\label{item:eminent-not-distinguished-superDa-1-gen}   
$\superda{\alpha}$ with diagrams 
\begin{align*}
&\xymatrix@C=20pt{\overset{q}{\circ} \ar@{-}[r]^{q^{-1}} &\overset{-1}{\circ}\ar@{-}[r]^{r^{-1}}\ &\overset{r}{\circ}}
&M&<\infty, \quad N, L=\infty;&
\\
&\xymatrix@C=20pt{\overset{q}{\circ} \ar@{-}[r]^{q^{-1}} &\overset{-1}{\circ}\ar@{-}[r]^{r^{-1}}\ &\overset{r}{\circ}}
&L&<\infty, \quad M,N=\infty;
\\
&\xymatrix@C=20pt@R=15pt{
&\overset{-1}{\underset{1}{\circ}}\ar@{-}[ld]_{q}\ar@{-}[rd]^{r} & 
\\
\overset{-1}{\underset{2}{\circ}} \ar@{-}[rr]^{s} & &\overset{-1}{\underset{3}{\circ}}} &M&<\infty,\quad N,L=\infty.
\end{align*}
\end{enumerate}
Here $q,r,s \in \ku^{\times}$ satisfy $qrs=1$ and $M \coloneqq \ord q$, 
$N \coloneqq \ord r$, $L \coloneqq  \ord s$. 
Moreover,  $\Ec(V)$  fits into an exact sequence of graded  Hopf algebras like
\eqref{eq:exact-sequence-prenichols}.
\end{theorem}

The defining relations, PBW basis and $\GK$ of the eminent algebras in the previous
theorem are given in \cite{campagnolo}*{\S4.2, \S4.3}.

See \cite{angiono-campagnolo}*{Theorem~4.8} for $V$ with
disconnected Dynkin diagram.

\subsection{Post-Nichols algebras of diagonal type}
 The post-Nichols algebra $\Lc(V) \coloneqq \widetilde{\Bc}(V^*)^d$
 is called the Lusztig algebra of $V$. 
 We know that $\Lc(V)$ is Noetherian if $\dim\toba(V)<\infty$ \cite{aarossi-mrl}*{Proposition 4.14}.
 This result can be generalized as follows.

 \begin{lemma}\label{prop:gener-Lusztig-Noetherian} 
 Let $V$ be a braided vector space of diagonal type such that $\GK\toba(V)< \infty$.
If
the connected components of the Dynkin diagram of $V$ satisfy the restrictions in
 Theorem~\ref{thm:eminent}, i.e., they are not as in \eqref{eq:exceptions},
 then the graded dual $\Ec(V)^d$  of the eminent pre-Nichols algebra  is Noetherian.
 \end{lemma}

  \noindent
  \emph{Sketch of the proof.} If $\dim \toba(V)< \infty$ and $\Ec(V) \simeq \wtoba(V)$, then this follows from \cite{aarossi-mrl}*{Proposition 4.14}. If $\dim \toba(V) =  \infty$,
   $\GK\toba(V)< \infty$ and $\Ec(V) \simeq \toba(V)$, 
   repeat the proof of \cite{aarossi-mrl}*{Proposition 4.14}.
  
  For the remaining  cases   \ref{item:eminent-not-distinguished-A2}, \dots, 
  \ref{item:eminent-not-distinguished-superDa-1-gen},
  when $\wtoba(B)$ (or $\toba(V)$) is not eminent, the proof of 
  \cite{aarossi-mrl}*{Proposition 4.14} could be adapted. 
  \qed

\medbreak
We leave to the reader the following question, interesting in its own right:

\begin{problem} Let $V$ be as in
\ref{item:eminent-not-distinguished-A2}, \dots, 
\ref{item:eminent-not-distinguished-superDa-1-gen}.
Give a presentation of the graded dual $\Ec(V)^d$ of the eminent pre-Nichols algebra of $V$;
can one conclude that it is Noetherian?
\end{problem}

\begin{theorem}\label{thm:Noetherian-finiteGK-diagtype}
Let $H$ be a pointed affine Hopf algebra with $\GK H < \infty$.
As usual, decompose $\gr H \simeq R \# \ku \Gamma$. Assume that

\begin{itemize}  [leftmargin=*]\renewcommand{\labelitemi}{$\circ$}
	\item $R$ is affine,
	
	\item $V \coloneqq R^1$ is of diagonal type, and
	
\item the connected components of the Dynkin diagram of $V$ satisfy the restrictions in Theorem \ref{thm:eminent},  i.e.,  are not as in \eqref{eq:exceptions}.
\end{itemize}
Then $H$ is Noetherian.
\end{theorem}

\medbreak
  \noindent
\emph{Proof assuming Lemma \ref{prop:gener-Lusztig-Noetherian}.}  We know that $\Gamma$ is nilpotent-by-finite.
Since $H$ is pointed and $R$ is affine, $\GK R < \infty$
by Theorem \ref{thm:zhuang}; \emph{a fortiori} $\GK \toba(V) < \infty$. Let $\Ec$ be the eminent pre-Nichols algebra
of $V$.  By Lemma~\ref{prop:gener-Lusztig-Noetherian}, we may apply
Lemma~\ref{lemma:eminent-Noetherian}
to conclude that $R$ is Noetherian.
Then Lemma \ref{lem:Noetherian-finiteGK} implies that $H$ is Noetherian.
\qed

\medbreak
Since our proof of Lemma \ref{prop:gener-Lusztig-Noetherian} is not complete, 
we will provide an alternative argument for Theorem \ref{thm:Noetherian-finiteGK-diagtype}
below Theorem~\ref{thm:core_noetherian}.

\section{The core and Noetherian graded  Hopf algebras}\label{sec:core}
In this section the field $\ku$ is arbitrary. 
The aim of this section is to provide a sufficient condition for the Noetherianity of connected graded braided Hopf algebras--see Theorem~\ref{thm:core_noetherian}--by means of the new notion of the \emph{core} of a connected graded braided Hopf algebra.

\medbreak
We fix a totally ordered abelian monoid (also called a tomonoid) $(\monoid,+,<)$. In particular, $a+c<b+c$ for any $a,b,c\in \monoid $ with $a<b$. We assume  that 
\begin{equation}
\label{eq:well_founded}
    \{\beta\in\monoid:\beta<\alpha\}\text{ is finite} 
\end{equation}
for any $\alpha\in\monoid$. 
\medbreak
% \begin{example*}
% If $\mathtt{G}$ is a subgroup of $\mathbb{R}$, then 
% $\monoid \coloneqq \mathtt{G}\cap\mathbb{R}_{\geq 0}$ is a totally ordered abelian monoid; 
% $\monoid$ is discrete exactly when it is discrete with the relative topology.
% In general a totally ordered abelian monoid embeds into a totally ordered abelian group, by the Grothendieck construction.
% \end{example*}

\begin{example*}
For each positive integer $\theta$, $\N_0^\theta$ is a totally ordered abelian monoid by 
letting $(n_1,\dots,n_\theta)<(m_1,\dots,m_\theta)$ if and only if
\begin{align*}
&\sum_{1\leq k\leq\theta} n_k<
\sum_{1\leq k\leq\theta} m_k\text{ or }\\
&\sum_{1\leq k\leq\theta} n_k=
\sum_{1\leq k\leq\theta} m_k\text{ and
$(n_1,\dots,n_\theta)<(m_1,\dots,m_\theta)$ lexicographically.}
\end{align*}
%such that the set $\{\beta\in\monoid:\beta<\alpha\}$ is finite for any $\alpha$. 

% If $\mathtt{G}$ is a subgroup of $\mathbb{R}$, then 
% $\monoid \coloneqq \mathtt{G}\cap\mathbb{R}_{\geq 0}$ is a totally ordered abelian monoid; 
% $\monoid$ is discrete exactly when it is discrete with the relative topology.
% In general a totally ordered abelian monoid embeds into a totally ordered abelian group, by the Grothendieck construction.
\end{example*}

\begin{remark*} By \cite{hs-book}*{Lemma~5.2.2}, the neutral element $\cero$ is the unique minimal element of $\monoid$. Hence each decomposition of some $\alpha \in \monoid$ as a sum of at least two non-zero elements consists of summands less than $\alpha$. By \emph{loc. cit.}, $\monoid$ is torsion-free, cancellative and positive.
\end{remark*}

% Note that $\monoid \times \N_{0}$ is a totally ordered monoid (e.g., lexicographically), 
% but not a discrete one.

\subsection{Hopf algebras graded by a monoid}
Let $J$ be a Hopf algebra and let $\monoid$ be the monoid as above. 
We denote by $\ydg{J}{\monoid}$ the category of 
$\monoid$-graded  \emph{locally finite} (i.e., with finite-dimensional $\monoid$-homogeneous components) objects in $\yd J$. The tensor product of 
$B=\bigoplus_{\alpha \in \monoid}B^\alpha$ and $C=\bigoplus_{\beta \in \monoid}C^\beta \in 
\ydg{J}{\monoid}$ is $\monoid$-graded by
\begin{align*}
(B \otimes C)^{\gamma} \coloneqq \bigoplus_{\alpha, \beta \in \monoid: \alpha + \beta = \gamma}B^\alpha \otimes C^\beta.
\end{align*}
According to the above remark and~\eqref{eq:well_founded},  
$B \otimes C \in \ydg{J}{\monoid}$ and $\ydg{J}{\monoid}$ becomes a tensor category, whose unit $\ku$ has degree $\cero$. In $\ydg{J}{\monoid}$, all finite-dimensional objects are rigid, but the category 
is generally not rigid.
Nevertheless, we may consider bialgebras and Hopf algebras in $\ydg{J}{\monoid}$.

\medbreak
Let  $C \in \ydg{J}{\monoid}$.
The (left, respectively right) $\monoid$-graded dual of $C$ is 
\begin{align*}
{}^{\leftdual}C &\coloneqq \bigoplus_{\alpha \in \monoid} {}^{*}(C^\alpha), &
&\text{respectively}& 
C^{\rightdual} \coloneqq \bigoplus_{\alpha \in \monoid} (C^\alpha)^{*}.
\end{align*}
These belong to $\ydg{J}{\monoid}$ but are not (left, respectively right) duals in the categorical sense.
Nevertheless, if $B$ is a bialgebra, or a Hopf algebra, in $\ydg{J}{\monoid}$,
then so are ${}^{\leftdual}B$ and $B^{\rightdual}$.

\begin{example*}
If $C\in \ydg{J}{\monoid}$, then 
$T(C)$ is  a Hopf algebra  in $\ydg{J}{\monoid}$.
Since the braiding preserves the $\monoid $-grading, so do the quantum symmetrizers.
Hence the Nichols algebra $\toba(C)$ is  a  Hopf algebra  in $\ydg{J}{\monoid}$. 
Note that  $\toba(C)^{\cero}=\ku1$ if and only if $C^{\cero}=\{0\}$, 
if and only if $T(C)^{\cero}=\ku1$.
\end{example*}

\medbreak 
Clearly, the Hopf algebras $T(C)$ and $\toba(C)$ are $(\monoid \times \N_{0})$-graded.

\medbreak
Generally, given an $(\monoid \times \N_{0})$-graded Hopf algebra $A$, we set
\begin{align*}
A^{\monoid,n} &\coloneqq\bigoplus_{\alpha \in \monoid} A^{\alpha,n} , & n &\in \N_0;&
A^{\alpha,\N_0} &\coloneqq\bigoplus_{n \in \N_{0}} A^{\alpha,n} , & \alpha &\in \monoid.
\end{align*}
Thus $A = \bigoplus_{n \in \N_{0}} A^{\monoid,n} =  \bigoplus_{\alpha \in \monoid} A^{\alpha,\N_{0}}$ 
is both an $\N_{0}$-graded and an $\monoid$-graded  Hopf algebra,
but is not locally finite in either sense.

\begin{remark*}
There is an isomorphism of Hopf algebras in $\ydg{J}{\monoid}$
\begin{align}\label{eq:M-dual}
{}^{\leftdual}(\toba(C)) \simeq \toba({}^{\leftdual}C). 
\end{align}
\end{remark*}
  Indeed, by definition of Nichols algebra we have
 \begin{align*}
 {}^{\leftdual}(\toba(C)) \simeq  \oplus_{n\in \N_0 }{}^{\leftdual}\left(\toba^n(C)\right) 
 \simeq  \ku \oplus {}^{\leftdual}C \oplus \dots \simeq \toba({}^{\leftdual}C).
 \end{align*}

\subsection{The core} We introduce now the main notion of this Section.
\emph{For the whole subsection we fix a Hopf algebra $B$ in $\ydg{J}{\monoid}$ such that 
$B^\cero=\ku 1$.} Then $B$ is connected, 
\begin{align*}
\ker \varepsilon &= \bigoplus_{\alpha > \cero} B^\alpha, &&\text{and} 
&\bigcap_{n\ge 1}(\ker \varepsilon)^n &=\{0\}.
\end{align*} 

\begin{definition} \label{def:core}
A \emph{core of $B$} is 
an object $C \in \ydg{J}{\monoid}$ 
such that $B$ and $\toba(C)$ are isomorphic  as objects in $\ydg{J}{\monoid}$ 
(not necessarily as Hopf algebras). 
\end{definition}

The grading by $\monoid$ and the assumptions on $\monoid$ itself are
an essential part of our definition of a core.
Indeed, note that in the category of vector spaces over characteristic $0$ fields, the Nichols algebra
of any finite-dimensional vector space $V$ is isomorphic to a polynomial ring in $\dim V$ indeterminates.
Hence a non-trivial affine Hopf algebra $A$ admits either no $V$ or infinitely many non-isomorphic $V$
such that $A\simeq \toba(V)$ as vector spaces.
In contrast, in the graded setting, 
we will prove an existence and uniqueness result in Theorem~\ref{thm:existence}.
Note however that variations of our definition e.g. towards suitably filtered Hopf algebras are reasonable.

\begin{lemma}
\label{lem:core_dual}
If $B$ has a core $C$, then ${}^{\leftdual}C$ is a core of  ${}^{\leftdual}B$. 
\end{lemma}

\begin{proof}
This follows from the fact \eqref{eq:M-dual}.     
\end{proof}

\begin{lemma}
\label{lem:unique_core}
The core of $B$, if it exists, is unique up to isomorphism as an object in $\ydg{J}{\monoid}$. 
\end{lemma}

\begin{proof}
Let $C=\bigoplus_{\alpha \in \monoid }C^\alpha $ and $D=\bigoplus_{\alpha \in \monoid}D^\alpha $ be cores of $B$. Note that \[
C^\cero=D^\cero=\{0\}
\]
since $B^\cero=\ku 1$.
By assumption, $\toba(C)$ and $\toba(D)$ are isomorphic in $\ydg{J}{\monoid}$. By comparing components we conclude that
\begin{align} \label{eq:gradediso}
&\bigoplus_{n\ge 1} \, \sum_{\alpha_1+\dots +\alpha_n= \alpha } C^{\alpha_1}\cdots C^{\alpha_n}
\simeq
\bigoplus _{n\ge 1} \, \sum _{\alpha_1+\dots +\alpha_n = \alpha } D^{\alpha_1} \cdots D^{\alpha_n}
\end{align}
in $\yd J$ for all $\alpha \in \monoid$. Note that on both sides of \eqref{eq:gradediso}
the object is isomorphic to $B^\alpha$ and hence finite-dimensional.
% We separate the terms with $n=1$:
% \begin{align*}
% &C^\alpha \oplus \bigoplus_{n\ge 2} \, \sum_{\alpha_1+\dots +\alpha_n= \alpha } C^{\alpha_1}\cdots C^{\alpha_n}
% \simeq
% D^{\alpha}\oplus \bigoplus _{n\ge 2}\, \sum _{\alpha_1+\dots +\alpha_n = \alpha } D^{\alpha_1} \cdots D^{\alpha_n}
% \end{align*}
% in $\yd J$ for all $\alpha \in \monoid $.
Let now $\alpha \in \monoid$ be such that $C^\beta \simeq D^\beta$ for all $\beta <\alpha$.
Fix a sum $\alpha_1+\cdots +\alpha_n=\alpha$ with $n\ge 2$ and non-zero summands $\alpha_i$'s. 
By assumption on $\monoid$, one has $\alpha_i<\alpha$ and hence $C^{\alpha_i} \simeq D^{\alpha_i}$ for all $i$.
Let
\[ C^{<\alpha}=\bigoplus_{\beta <\alpha}C^\beta, \qquad D^{<\alpha}=\bigoplus_{\beta <\alpha}D^\beta. \]
Then, by functoriality of $\toba$, the summands for $n\ge 2$ in \eqref{eq:gradediso} are the summands
of $\toba (C^{<\alpha})$ and $\toba(D^{<\alpha})$
of degree $\alpha$, and hence isomorphic. We conclude that $C^\alpha \simeq D^\alpha $.
By induction on $\alpha$ it follows that $C\simeq D$ in $\ydg J \monoid$.
\end{proof}

For any  Hopf algebra $A$ in $\yd J$, let $\gr_{\corad}A$ be the $\N_0$-graded  Hopf algebra
 in $\yd J$ associated to the coradical filtration of $A$. Similarly, let 
$\gr^{\ker\varepsilon}A$ be the $\N_0$-graded  Hopf algebra  in $\yd J$
associated to the descending filtration of $A$ by powers of $\ker\varepsilon$.

Let $A$ be an $\monoid$-graded  Hopf algebra in $\yd J$ such that
the coradical $A_0$ is an $\monoid$-graded subobject. Then all terms $(A_n)_{n \geq 0}$ 
of the coradical filtration are $\monoid$-graded subobjects 
and $\gr_{\corad}A$ is  an $(\monoid \times\N_0)$-graded  coalgebra.
If in addition $A_0$ is a Hopf subalgebra in $\yd J$, then 
$\gr_{\corad}A$ is  an $(\monoid \times\N_0)$-graded  Hopf algebra  in $\yd J$.
\begin{lemma} \label{lem:injections}
Let $A$ be an $(\monoid \times\N_0)$-graded  Hopf algebra in $\yd J$ such that
\begin{align}\label{eq:A-connected}
A^{\monoid,0}=\ku 1.
\end{align}
Then there exists an injective morphism of $\monoid$-graded objects in $\yd J$
\begin{align}\label{eq:injections}
A^{\monoid,1}\hookrightarrow (\gr_{\corad} A)^{\monoid,1}.
\end{align}
\end{lemma}

\begin{proof}
Since \eqref{eq:A-connected} says that $A$ is connected, it follows from the previous discussion that $\gr_{\corad}A$ is  an $(\monoid \times\N_0)$-graded  Hopf algebra  in $\yd J$. 
Now, the elements of $A^{\monoid,1}$ are primitive, hence the (injective)  natural map 
\eqref{eq:injections}.
\end{proof}

Let $A$ be an $\monoid$-graded  Hopf algebra in $\yd J$. 
Then the powers of the augmentation ideal are $\monoid$-graded subobjects 
and $\gr^{\ker\varepsilon}A$ is  an $(\monoid \times\N_0)$-graded  algebra.
If \eqref{eq:A-connected} holds, then 
$A$ is connected  and $\gr^{\ker\varepsilon}A$ is  an $(\monoid \times\N_0)$-graded  Hopf algebra  in $\yd J$.

\begin{lemma}
\label{lem:surjection}
Let $A$ be an $(\monoid\times \N_0)$-graded  Hopf algebra in $\yd J$ such that
\eqref{eq:A-connected} holds.
Then there exists a surjective morphism  of $\monoid$-graded objects in $\yd J$
\begin{align*}
(\gr^{\ker\varepsilon}A)^{\monoid,1}\twoheadrightarrow 
A^{\monoid,1}
\end{align*}
\end{lemma}

\begin{proof}
By \eqref{eq:A-connected},  $A^{\monoid,0}=\ku 1$,   
$\ker\varepsilon=\bigoplus_{n\geq1}A^{\monoid,n}$ and
$(\ker\varepsilon)^2\subseteq \bigoplus_{n\geq2}A^{\monoid,n}$.
This implies the claim. 
\end{proof}

It is well-known that a connected coradically graded Hopf algebra $A$ in $\yd J$ is the Nichols algebra of its degree one part
if and only if $A$ is generated as an algebra by its degree one part.
The proof of Theorem~\ref{thm:existence} below will rely on the following weak graded version of this claim.

\begin{lemma}
\label{lem:braidedgraded}
    Let $A$ be an $(\monoid \times \N_0)$-graded Hopf algebra in $\yd J$ such that $A^\cero =\ku 1$, the $\monoid$-grading is locally finite, and the $\N_0$-grading of the coalgebra $A$ is strict.
    Let $V=A^{\monoid,1}$.
    Then $\toba(V)$ is $(\monoid \times \N_0)$-graded.
    Moreover, for any non-zero $\alpha \in \monoid$, the natural graded embedding $V\to A$ induces a graded isomorphism
    %\[ A^{\alpha,\N_0}\simeq \toba(V)^{\alpha ,\N_0} \]
        \[ \bigoplus _{0<\beta\le \alpha} A^{\beta,\N_0}\simeq \bigoplus_{0<\beta \le \alpha}\toba(V)^{\beta ,\N_0} \]
    in $\yd J$ if and only if
    $\bigoplus_{0<\beta\le \alpha}A^{\beta,1}\simeq \bigoplus_{0<\beta\le \alpha}
    \left(A/(\ker \varepsilon_A)^2\right)
    ^{\beta ,\N_0}$ in $\yd J$.
\end{lemma}

\begin{proof}
    The assumptions of the lemma imply that the natural graded embedding of $V$ into $A$ induces a natural injective $(\monoid \times \N_0)$-graded Hopf algebra morphism $f:\toba(V)\to A$.
    Hence the only if part of the claim holds. Indeed, the isomorphisms $A^{\beta,n}\simeq \toba(V)^{\beta,n}$ for $n\ge 1$ and $0<\beta\le \alpha$ and the local finiteness of the $\monoid$-grading imply that $A^{\beta ,n}=(V^n)^{\beta,n}\subseteq ((\ker \varepsilon_A)^2)^{\beta,n}$ for all $n\ge 2$ and $0<\beta \le \alpha$.

    We prove now the if part of the Lemma. Assume that $A^{\beta ,n}\not \simeq \toba (V)^{\beta ,n}$ for some $n\ge 2$ and $0<\beta \le \alpha$. We may assume that $n$ is minimal under this condition. Then
    $A^{\monoid, k}=V^k=(\ker \varepsilon)^{\monoid,k}$ for all $1\le k<n$. Therefore
    \[ \big((\ker \varepsilon)^2\big)^{\monoid,n}=
    \sum_{k=1}^{n-1} A^{\monoid,k}A^{\monoid,n-k}=V^{\monoid ,n}, \]
    and hence
    \[ \left(A/(\ker \varepsilon)^2\right)^{\beta ,n}\simeq A^{\beta ,n}/((\ker \varepsilon)^2)^{\beta ,n}
    \simeq A^{\beta ,n}/(V^n)^{\beta,n}
    \not \simeq 0,
    \]
    which confirms the if part of the Lemma.
\end{proof}

\begin{theorem}
\label{thm:existence}
Let $B$ be a Hopf algebra  in $\ydg{J}{\monoid}$ such that 
$B^\cero=\ku 1$.
Assume that  $B$ is a semisimple object in $\yd J$. Then $B$ admits a  core
that is unique up to isomorphism. 
\end{theorem}

\begin{proof}
The uniqueness follows from Lemma~\ref{lem:unique_core}.
Regarding the existence, consider the family $(B_n)_{n\ge 1}$ of  Hopf algebras in $\yd J$, where
\begin{enumerate}[leftmargin=5ex,label=\rm{(\roman*)}]
\item $B_1=B$,

\medbreak
\item for each $n\ge 1$, $B_{2n}=\gr_{\corad}B_{2n-1}$,

\medbreak
\item for each $n\ge 1$, $B_{2n+1}=\gr^{\ker \varepsilon}B_{2n}$.
\end{enumerate}

\medbreak
Since $B$ is $\monoid$-graded, it follows that $B_n$ is $(\monoid\times\N_0)$-graded for all $n\geq2$, and that it is isomorphic to $B$ as an $\monoid$-graded object in $\yd J$.
Moreover, 
$B_n^{\monoid ,0}=\ku 1$ for all $n\ge 2$ since $B^\cero=\ku 1$.

For all $n\ge 2$ and $\alpha \in \monoid $ let $C_n^\alpha = B_n^{\alpha, 1}$ be the homogeneous component of $B_n$ of degree $(\alpha,1)$, and let 
$C_n=\bigoplus_{\alpha \in \monoid}C_n^\alpha = B_n^{\monoid, 1}$.
By Lemmas~\ref{lem:injections} and \ref{lem:surjection} and the semisimplicity of $B$, there exist injective morphisms $\iota_n:C_n\to C_{n+1}$ of $\monoid$-graded objects in $\yd J$. Since the $\monoid$-grading of $B$ is locally finite, the sequence $(C_n^\alpha )_{n\ge 2}$ stabilizes for each $\alpha \in \monoid$ at some $C^\alpha \in \yd J$. Let \[ C= \bigoplus_{\alpha \in \monoid}C^\alpha. \]
We are going to prove that $C$ is a core of $B$ by proving that $B^\alpha \simeq \toba(C)^\alpha$ for all $\alpha \in \monoid$ using  Lemma~\ref{lem:braidedgraded}. The latter claim for $\alpha=0$ holds by assumption.

Let $\alpha \in \monoid$ with $\alpha>0$, and let $n\ge 1$ be such that $C^\beta\simeq C_{2n}^\beta $ for all $\beta \le \alpha$. The existence of such $n$
is guaranteed by \eqref{eq:well_founded} and the above analysis of $C$.
%By construction, $B_{2n}$ is coradically graded. 
%Hence the subalgebra $\ku\langle C_{2n}\rangle $ of $B_{2n}$ generated by $C_{2n}$ is isomorphic to $\toba(C_{2n})$ as an $(\monoid \times \N_0)$-graded  Hopf algebra in $\yd J$.
Then $B^\beta\simeq B_{2n}^{\beta ,\N_0}$
in $\yd J$ for all $\beta$,
and $B_{2n}$ is coradically graded by definition. Moreover, $B_{2n}^{\beta,1}\cap (\ker \varepsilon_{B_{2n}})^2=0$ and
\[ B_{2n}^{\beta,1}\simeq C_{2n}^\beta \simeq C_{2n+1}^\beta \simeq B_{2n+1}^{\beta, 1}
\simeq (\ker \varepsilon_{B_{2n}} /(\ker \varepsilon_{B_{2n}})^2)^{\beta ,\N_0}\]
for all $\beta\le \alpha $ by construction.
Note that $B_{2n}^{\beta,\N_0}=(\ker \varepsilon_{B_{2n}})^{\beta,\N_0}$ for all $\beta >0$. Therefore Lemma~\ref{lem:braidedgraded} applied to $A=B_{2n}$
implies that
\[ B_{2n}^{\beta ,\N_0}\simeq \toba(B_{2n}^{\monoid,1})^{\beta,\N_0}\]
for all $0<\beta\le \alpha$ as an $(\monoid\times \N_0)$-graded object in $\yd J$.
Now recall that $B_{2n}^{\alpha ,\N_0}\simeq B^\alpha$ and that $B_{2n}^{\beta ,1}\simeq C_{2n}^\beta \simeq C^\beta$ for all $\beta \le \alpha$.
Thus $B^{\alpha }\simeq \toba(C)^\alpha$ in $\yd J$, as announced above, and hence $C$ is a core of $B$.
\end{proof}

\begin{cor}
\label{cor:eminent} 
Let $V$ be a braided vector space of diagonal type such that $\GK\toba(V)< \infty$.
Assume that the connected components of the Dynkin diagram of $V$ satisfy the restrictions in
Theorem~\ref{thm:eminent}. Then the eminent pre-Nichols 
algebra $\Ec$ of $V$ has a finite-dimensional core $C$ and 
the root system of $\toba(C)$ is finite. 
\end{cor}

\begin{proof}
By Theorems \ref{thm:eminent}, or \ref{thm:campagnolo}, we know that there is an exact sequence like \eqref{eq:exact-sequence-prenichols}. By \cite{AAGMV}*{Proposition 3.6 (d)} we get a right
$\hZ $-linear isomorphism $\Ec(V) \simeq \toba(V) \otimes \hZ$ of objects in $\ydg J{\N_0}$ for the group ring $J$ of a suitable abelian group. Moreover, the braiding of $\Ec (V)$ is of diagonal type, and hence $\Ec(V)$ is semisimple in $\yd J$.
By Theorem~\ref{thm:existence}, $\Ec$ has a unique core. Since $V\subseteq \Pc(\Ec)$, the core of $\Ec$ contains $V$. Let $W\subseteq \Ec$ be the span of the homogeneous $q$-polynomial generators of $\hZ$.
By inspection of the individual examples, it is straightforward to check that  
for each $q$-polynomial generator $z$ of $\hZ$ one has
\[ c(z\otimes z)=z\otimes z,\quad c^2(z\otimes v)=z\otimes v \]
for all $v\in V$. 
Therefore
\[ \toba (V\oplus W) \simeq \toba(V)\otimes \toba(W) \simeq \toba (V)\otimes \hZ 
\]
as $\N_0$-graded Yetter-Drinfeld modules. 
Hence $V\oplus W$ is a core of $\Ec$. The root system of $\toba(V\oplus W)$ is the union of the root system of $V$ and the root system of $W$, and hence is finite.
\end{proof}

\subsection{Noetherian Hopf algebras}

\begin{definition}
Let $\searrow $ be the transitive closure of the relation 
on the class of connected $\monoid$-graded  Hopf algebras in $\yd J$
induced by the relations
\[
A\searrow B,
\]
where $A$ is an $\monoid$-graded Hopf algebra in $\yd J$ with $A^\cero=\ku 1$ and 
$B$ is isomorphic as an $\monoid$-graded  Hopf algebra in $\yd J$ to
$\gr_{\corad}A$ or to $\gr ^{\ker \varepsilon} A$; see the discussion above 
Lemma~\ref{lem:injections}. We call the relation $\searrow$ 
the \emph{graded degeneration order}. 
\end{definition}

Here is a non-trivial example of the graded degeneration order.

\begin{example*}
Let $A=U_q^+(\mathfrak{sl}_3)$,  
where $q\in\C$ is a primitive third root of 1, with generators
$x_1$ and $x_2$ and braiding matrix $\begin{pmatrix}q^2&q^{-1}\\q^{-1}&q^2\end{pmatrix}$. Then 
$A$ is the \emph{distinguished }
pre-Nichols algebra of $V=\mathrm{span}_\ku \{x_1,x_2\}$. In particular, $A$ is $\N_0$-graded, where $x_1$ and $x_2$ have degree one. Hence $\gr^{\ker\varepsilon}A\simeq A$. 
 Let 
\[
x_{12}=x_1x_2-q^{-1}x_2x_1.
\]
Then in $A$ the relations
\[ x_1x_{12}=qx_{12}x_1,\quad 
x_{12}x_2=qx_2x_{12}\]
hold. Moreover,
$\{x_2^{a}x_{12}^bx_1^{c}:a,b,c\in\N_0\}$ is a PBW basis 
of $A$. The elements $x_1^3$ and $x_2^3$ are primitives in $A$ and 
\[
\Delta(x_{12}^3)=x_{12}^3\otimes 1+\lambda x_1^3\otimes x_2^3+1\otimes x_{12}^3
\]
for some $\lambda\in\C$. Moreover, 
$x_2^3, x_{12}^3, x_1^3$ are central in $A$. We write $y_2$, $y_{12}$, $y_1$ for the generators of $\gr_{\corad}A$  corresponding to $x_2^3$, $x_{12}^3$ and $x_1^3\in A$, respectively. It follows that
\[ 
\{y_2^{m_2}y_{12}^{m_{12}}y_1^{m_1}x_2^{n_2}x_{12}^{n_{12}}x_1^{n_1}:m_2,m_{12},m_1\in\N_0, 0\le n_2,n_{12},n_1\le 2\}
\]
is a basis of $\gr_{\corad}A$,
\[
\deg(y_2^{m_2}y_{12}^{m_{12}}y_1^{m_1}x_2^{n_2}x_{12}^{n_{12}}x_1^{n_1})=m_2+2m_{12}+m_1+n_2+2n_{12}+n_1,
\]
and
$x_2^3=x_{12}^3 =x_1^3=0$ in $\gr _{\corad} A$.

Note that $y_{12}\not\in(\ker\varepsilon)^2$. Then 
$y_{12}\in \gr^{\ker\varepsilon}\gr_{\corad}A$ is a primitive generator. In particular, we observe that 
\[
\gr^{\ker\varepsilon}\gr_{\corad}A\not\simeq \gr_{\corad}A\simeq
\gr_{\corad}\gr^{\ker\varepsilon}A,
\]
where the isomorphism $A\simeq \gr^{\ker \varepsilon}A$ holds since $A$ is a pre-Nichols algebra.
\end{example*}

\begin{theorem} \label{thm:grdegorder} Let $B$ be a Hopf algebra  in $\ydg{J}{\monoid}$ such that 
$B^\cero=\ku 1$.
Assume that $B$ is a semisimple object in $\yd J$ with a finite-dimensional core $C$. 
Then $B\searrow\toba(C)$. 
\end{theorem} 

\begin{proof}
Follow the proof of Theorem~\ref{thm:existence}. Since $\dim C<\infty$, the sequence $(B_n)_{n\ge 1}$ in that proof stabilizes after some $n\ge 1$.
\end{proof}

\begin{prop}
\label{pro:semisimple}
Assume that $B$ is a semisimple object in $\yd J$ with a finite-dimensional core $C$. 
If $\toba(C)$ is Noetherian, then 
$B$ is Noetherian. 
\end{prop}

\begin{proof}
By Theorem~\ref{thm:grdegorder}, $B\searrow \toba(C)$. For all $\monoid$-graded  Hopf algebras $A$ in $\yd J$ with $A^\cero=\ku 1$ the Noetherianity of any of the graded  Hopf algebras $\gr _{\corad}A$ and $\gr ^{\ker \varepsilon }A$ implies that $A$ is Noetherian.
This implies the claim.
\end{proof}

\begin{theorem}
\label{thm:core_noetherian}
Let $A$ be a locally finite $\monoid$-graded  Hopf algebra of diagonal type with $A^\cero=\ku 1$. Assume that 
$A$ has a finite-dimensional core $C$, and that the root system of $\toba(C)$ is finite. Then 
$A$ is Noetherian. 
\end{theorem}

\begin{proof}
By assumption, $A$ is a semisimple object in $\ydg J \monoid$ for an appropriate group ring $J$. Since 
the root system of $\toba(C)$ is finite, by Corollary~\ref{cor:noeth-implies-finiteGK}, $\toba(C)$ is Noetherian. 
The claim follows from Proposition~\ref{pro:semisimple}.
\end{proof}

Now we can give a proof of Theorem~\ref{thm:Noetherian-finiteGK-diagtype} without 
using Lemma \ref{prop:gener-Lusztig-Noetherian}. 

\begin{proof}[Proof of Theorem~\ref{thm:Noetherian-finiteGK-diagtype}] 
We know that $\Gamma$ is nilpotent-by-finite.
Since $H$ is pointed and $R$ is affine, $\GK R < \infty$
by Theorem \ref{thm:zhuang}; \emph{a fortiori} 
\[
\GK \toba(V) < \infty.
\]
Let $\Ec$ be the eminent pre-Nichols algebra
of $V$. By Corollary~\ref{cor:eminent}, $\Ec$ has a finite-dimensional core $C$ and $\toba(C)$ has a finite root system. By Theorem~\ref{thm:core_noetherian}, $R$ is Noetherian (and hence \ref{item:GK-R-finita} implies \ref{item:R-Noetherian}).   
Then Lemma \ref{lem:Noetherian-finiteGK} implies that $H$ is Noetherian. 
\end{proof}

Now we formulate a partial generalization of 
Corollary~\ref{cor:noeth-implies-finiteGK}.

\begin{cor}
\label{cor:core_noetherian}
Let $A$ be a locally finite $\monoid$-graded  Hopf algebra of diagonal type with $A^\cero=\ku 1$. Assume that $A$ has a finite-dimensional core $C$. Then $A$ has finite $\GK$ if and only if the root system of $\toba (C)$ is finite. Moreover, if $\GK A<\infty$ then $A$ is Noetherian.
\end{cor}

\begin{proof}
Since $A$ has a finite-dimensional core $C$, both $A$ and $\toba(C)$ are affine. Thus $A$ and $\toba (C)$ have the same $\GK $, which only depends on the dimensions of the homogeneous components. Thus the first part of the claim follows from Theorem~\ref{thm:finiteGK}.

Thus the second part of the claim follows from Theorem~\ref{thm:core_noetherian}.
\end{proof}

% \begin{question}\label{question:noeth_implies_finiteGK_diag} Does in Corollary~\ref{cor:core_noetherian} follow from Noetherianity of $A$ that $\GK A<\infty$?
% \end{question}

\section{Pointed Hopf algebras: what remains to be done}\label{sec:remains}
We keep the previous notation, namely  $\Gamma$ is a polycyclic-by-finite group,
$V \in \yd{\ku \Gamma}$
and $R$ is a post-Nichols algebra of  $V$. 

\subsection{Post-Nichols algebras}
A few delicate points need to be clarified.
We already stated Question \ref{question:graded-Noetherian}.

\begin{problem} 
Could Question \ref{question:graded-Noetherian} be answered, 
assuming that Conjecture \ref{conj:sierra-walton} holds? 
\end{problem}

\begin{problem} \label{problem:R-affine}
Assume that $R$ is the diagram of a pointed Noetherian affine Hopf algebra $H$. 
Under what conditions is $R$ affine? Is $\car \ku = 0$ enough?
\end{problem}

\begin{example}\label{exa:R-not affine} Suppose that $\car \ku = p > 0$.
Let $H = \ku[T]$ be the polynomial algebra with the usual Hopf algebra structure, i.e., 
$T$ is primitive. Then $\gr H = R$ is neither affine nor Noetherian.
\end{example}

\pf Clearly, $H$ is pointed,  $G(H)$ is trivial and the space of primitives is
\begin{align*}
\Pc(H) &= \sum_{j\in \N_{\geq 1}} \ku T^{P^j}.
\end{align*}
Set $y_j = $ class of $T^{P^j}$ in $H_1/H_0$. Then 
$(y_j)_{j\in \N_{\geq 1}}$ is a basis of $\gr^1 H$. We claim that $\gr H$ is not affine.

Indeed, suppose that $(z_i)_{i \in \I}$ is a finite set of generators of $R$;
we may assume that the $z_i$'s are homogeneous, hence those of them having degree 1 will form a
basis of $\gr^1 H$, a contradiction. Analogously, the ideal 
\[
R^+ = \bigoplus_{j\in \N_{\geq 1}} R^j
\]
could not be finitely generated, so $R$ is not Noetherian.
\epf

\begin{example}\label{exa:R-not affine2} Here $\car \ku = 0$.
 Let $H = T(V)$ be the tensor algebra of a finite dimensional vector space $V$,
 with the usual Hopf algebra structure, i.e., the elements of
 $V$ are primitive. Then, $H \simeq U(L(V))$, the enveloping algebra of the
 free Lie algebra generated by $V$. Clearly $H$ is affine; here the coradical filtration
is the PBW filtration and $\gr H \simeq S(L(V))$ is not affine provided that $\dim V > 1$.
\end{example}

\begin{problem}
If $V$ admits an eminent pre-Nichols algebra $\Ec$,
is its graded dual $\Ec^d$ Noetherian? 
\end{problem}

\begin{problem}
Let $V$ be of Cartan type either $A_{\theta}$ or  $D_{\theta}$ with label $q=-1$.
Does it admit an eminent pre-Nichols algebra?
\end{problem}

Corollary \ref{cor:noeth-implies-finiteGK} establishes 
\ref{item:R-Noetherian} $\implies$ \ref{item:GK-R-finita}
for Nichols algebras of diagonal type. 

\begin{problem}
Do Noetherian post-Nichols algebras of diagonal type have finite $\GK$? 
\end{problem}

\subsection{Post-Nichols  algebras over abelian groups}
The classification of the Nichols algebras over abelian groups with finite $\GK$
is not complete. The significant class corresponding to direct sums of blocks and points
was determined in \cite{AAH2}; but this does not exhaust the possible ones.
See \cite{AAH2}*{Section 9} and \cite{AAM}.
All the Nichols algebras listed there are Noetherian.
However the pre-Nichols algebras with finite $\GK$
are not known except for the Jordan plane--which is
Noetherian and the only pre-Nichols with finite $\GK$; see \cite{AAH1}. 
This represents then a source of problems, starting with:

\begin{problem}
Which Nichols algebras of blocks are Noetherian? 
\end{problem}

\subsection{Nilpotent groups} Fix a finitely generated nilpotent group $\varGamma$.
Given $M\in \yd{\ku \varGamma}$ such that $\toba(M)$ is Noetherian, then $M$ is finite dimensional
by Lemma \ref{lemma:graded-Noetherian}, hence the support $\mathcal O$ of $M$ is a finite union
of finite conjugacy classes. There is also a reduction for finite $\GK$ as we now recall.  

\begin{theorem}\label{thm:infinite-conjugacy-class} \cite{andrus-nilpotent}*{Theorem 2.6}.
Let $M\in \yd{\ku \varGamma}$ such that $\mathcal O \coloneqq \supp M$
is a single conjugacy class.
If $\GK \toba(M) \# \ku \varGamma < \infty$, then $\mathcal O$ is finite.
\end{theorem}

The study of Nichols algebras over torsion free nilpotent groups reduces to the abelian case.

\begin{theorem}\cite{andrus-nilpotent}*{Theorem 3.5}\label{thm:main-nichols-nilpotent-torsionfree}
Assume that $\varGamma$ is torsion-free.  
Let $M \in \yd{\ku \varGamma}$ be semisimple of finite length.
Then 
\begin{align*}
\GK \toba(M) \# \ku \varGamma < \infty
\end{align*}
if and only if $\supp M \subset Z(\varGamma)$ and $M$ is as
in  \cite{andrus-nilpotent}*{Proposition 3.4}. \qed
\end{theorem}

By inspection, we conclude that the Nichols algebras $\toba(M)$
with $M$ as in  \cite{andrus-nilpotent}*{Proposition 3.4} are Noetherian.

\begin{problem}
Assume that $\varGamma$ is torsion-free.  
Let $M \in \yd{\ku \varGamma}$ be semisimple of finite length
such that $\toba(M)$ is Noetherian. Is $\GK \toba(M) < \infty$?
\end{problem}

We refer to \cite{andrus-nilpotent}*{\S3} for a discussion of the
Nichols algebras in $\yd{\ku \varGamma}$ with finite $\GK$ under the assumption that
the torsion subgroup of $\varGamma$ has order coprime to $6$.
This discussion presupposes a positive answer to 
\cite{andrus-nilpotent}*{Conjecture 1.13}---a conjecture that seems presently out of reach.
Nevertheless, all Nichols algebras with finite $\GK$ appearing in these partial results
are  Noetherian.

\section{Beyond pointed}\label{sec:chevalley}

\subsection{Hopf algebras with the Chevalley property} 

Naturally, one may wonder to what extent the same thoughts for 
pointed Hopf algebras can be extrapolated to Hopf algebras with the 
Chevalley property, i.e., whose coradical is a Hopf subalgebra. 
We start recalling a basic result of Takeuchi.

\begin{theorem} \cite{Takeuchi-1972}*{Theorem 3.2}
Let  $N$ be a Hopf subalgebra of  a Hopf algebra $J$
such  that $J_0  \subseteq G(J)N$. Then
$J$ is a faithfully flat $N$-module.
\end{theorem}

\begin{cor}\label{cor:coradical-noetherian}
The coradical of a Noetherian 
Hopf algebra  with the Chevalley property is Noetherian.
\end{cor}

Thus,  Question  \ref{question:Noetherian-finiteGK}
might be translated as follows:

\begin{question}\label{question:Noetherian-GKdim-gral}
If $M$ is  a Noetherian Hopf algebra with the Chevalley property, 
and $\GK M_0$ is finite, is necessarily $\GK M < \infty$?
\end{question}

Since  Question \ref{question:affine+finiteGK-Noetherian} 
also stands, one is motivated to consider Conjecture \ref{conj:nichols-noeth-gral} with the additional assumption that $J = K$
is cosemisimple.

\medbreak
Now, let $A$ be an algebra with 
$\GK A<\infty$.
A finite dimensional subspace $V\subseteq A$ is
\emph{GK-deterministic} if
\begin{align*}
\GK A=\lim_{n \to \infty} \log_n \dim \sum_{0\le j\le n}V^j.
\end{align*}

\begin{lemma}\label{lemma:GKdim-smashproduct} \cite{AAH2}*{Lemma 2.3.1} Let 
$T$ be a $J$-module algebra such that the action 
of $J$ on $T$ is locally finite (this assumption  cannot be dropped).
Then 
\begin{align}\label{item:GKdim-smashproduct}
\GK T\rtimes J \le \GK T+\GK J.
\end{align}
If  either $J$ or $T$ have a GK-deterministic subspace, then 
the equality holds.
\end{lemma}

\begin{remark}\label{rem:Noetherian-DME-GKdim-chev}
Let $M$ be a Noetherian Hopf algebra with the Chevalley property
such that $K \simeq M_0$ satisfies $\GK K$ is finite. 
If $\gr\,M$ is Noetherian
and \ref{item:R-Noetherian-gral} implies \ref{item:GK-R-finita-gral}, 
then $\GK M < \infty$.
\end{remark}

\pf
Let  $\gr M\simeq R \# K$ be as above; it  is Noetherian
by assumption.
We infer from Lemma \ref{lemma:pointed corad-graded}
that $R$ is left Noetherian. 
Hence $\dim V < \infty$ by Lemma~\ref{lemma:graded-Noetherian}.
Since we assume that
\ref{item:R-Noetherian-gral} implies \ref{item:GK-R-finita-gral}, $\GK R < \infty$.
By Lemma \ref{lemma:GKdim-smashproduct}, 
$\GK M =  \GK R + \GK K < \infty$.
\epf

We briefly discuss the mirror situation to the previous Remark.
Namely, let $M$ be a Hopf algebra with the Chevalley property; then $\gr M \simeq R \# K$ as above. 
Assume that 
$\GK M < \infty$ and that the diagram $R$ is affine. 
By Proposition \ref{prop:kl-graded-filtered} $\GK \gr M < \infty$; a fortiori, 
$\GK R < \infty$ and $\GK K  < \infty$. 
If \ref{item:GK-R-finita-gral} implies \ref{item:R-Noetherian-gral}, 
then $R$ is Noetherian. So we are lead to the following two questions; the first is a particular case of
Question \ref{question:affine+finiteGK-Noetherian}.

\begin{question}\label{question:coss-finiteGK-implies Noeth}
Let $K$ be an affine \emph{cosemisimple} Hopf algebra such that $\GK K < \infty$; does it follow
that $K$ is Noetherian? 
\end{question}

\begin{question}\label{question:bosonization-Noeth-implies Noeth}
Let $K$ be a cosemisimple Hopf algebra, let $V\in \yd{K}$ and let $R$ be an affine post-Nichols algebra of $V$.
Assume that $K$ and $R$ are Noetherian. Does it follow that $R \# K$ is Noetherian?
\end{question}

\subsection{The standard filtration}

It was proposed to study more general Hopf algebras
by looking at the Hopf coradical, i.e., the
subalgebra generated by the coradical, and
the  standard filtration,  whose terms are 
iterative wedge operations of the Hopf coradical;
see \cite{andrus-cuadra}.
Although the steps in this approach are notoriously difficult, 
we record a first needed result.

\begin{cor}
The subalgebra generated by the coradical of a Noetherian 
Hopf algebra  is Noetherian.
\end{cor}

\subsection*{Acknowledgements} 
Part of the work of N. A. was done during a long term  visit  to the 
Shenzhen International Center for Mathematics at the Southern University of Science and Technology; he thanks Efim Zelmanov and Slava Futorny for the warm hospitality.

\medbreak
Part of the work  of N. A. was carried out while he held the position of Visiting Professor at the Vrije Universiteit Brussel (2025-2026). He thanks Leandro Vendramin 
for his friendly hospitality.

\medbreak
The final writing of this paper was done during a visit of N. A. to the Sydney Mathematical Research
Institute (SMRI) in April and May 2026. He thanks Gus Lehrer, Geordie Williamson and Ruibin Zhang
for the invitation and  warm welcome.

%\vspace{1ex}
\bibliographystyle{abbrv}
\bibliography{noeth-pointed.bib}

\end{document}